\documentclass[11pt,a4paper,reqno]{amsart}
\usepackage[margin=2cm,a4paper]{geometry}

\usepackage{amsmath}
\usepackage{amsfonts}
\usepackage{amssymb}
\usepackage{amsthm}

\newtheorem{theorem}{Theorem}[section]

\newtheorem{lemma}{Lemma}

\newtheorem{proposition}{Proposition}[section]

\title{Serrin-type Overdetermined problems in $\mathbb H^n$ }
\author{Zhenghuan Gao}\author{Xiaohan Jia}\author{Jin Yan}
\address{Department of Mathematics, Shanghai University, Shanghai, 200444, China.}\email{gzh@shu.edu.cn}
\address{School of Mathematical Sciences, Xiamen University, Xiamen, 361005, China.}\email{jiaxiaohan@xmu.edu.cn}
\address{School of Mathematical Sciences, University of Science and Technology of China, Hefei, 230025,  China.}\email{yjoracle@mail.ustc.edu.cn}

\begin{document}

\begin{abstract}
In this paper, we prove the symmetry of the solution to overdetermined problem for the equation $\sigma_k(D^2u-uI)=C_n^k$ in hyperbolic space. Our approach is based on establishing a Rellich-Pohozaev type identity and using a P function. Our result generalizes the overdetermined problem for Hessian equation in Euclidean space.
\end{abstract}


\maketitle

\section{Introduction}
In the seminal paper \cite{Serrin1971} Serrin established the symmetry of the solution to
\begin{equation}\label{laplace}
    \Delta u=n
\end{equation}
in a bounded $C^2$ domain $\Omega\subset\mathbb R^n$ with 
\begin{equation}\label{bdrycd}
u=0 \quad\text{and}\quad u_\gamma=1\quad\text{on }\ \partial\Omega,
\end{equation} 
where $\gamma$ is the unit outer normal to $\partial\Omega$. If $u\in C^2(\overline\Omega)$ is a solution to \eqref{laplace} and \eqref{bdrycd}, then $u=\frac{|x|^2-1}2$ upto a translation and $\Omega$ is the unitary ball. The proof is based on the method of \textit{moving planes} and it can be applied to more general uniformly elliptic equations. In \cite{Weinberger1971} Weinberger provided an alternative proof by using maximum principle for P function and a Rellich-Pohozaev type identity.

There have been many generalizations of Serrin and Weinberger's work to  quasilinear elliptic equations (see e.g. \cite{FK2008,FGK2006,Garofalo1989} and reference therein ) and fully nonlinear equations  such as Hessian equation and Weingarten curvature equation (see e.g. \cite{Salani2008,Jia2020,Bao2014}). 
In Euclidean space, the overdetermined boundary problem for $\sigma_k(D^2u)=C_n^k$ was studied in  \cite{Salani2008} by using a Rellich-Pohozaev type identity and some geometric inequalities and was also dealt in \cite{Bao2014} by using method of moving planes. Using the P function $P=|Du|^2-2u$ as mentioned in \cite{Ma1999, Weinberger1971} we can give an alternative proof which is parallel to Weinberger's.
\begin{theorem}[\cite{Salani2008}]\label{2thm}
Suppose $\Omega\subset \mathbb R^n$ is a $C^2$ bounded domain and $u\in C^3(\Omega)\cap C^2(\overline\Omega)$ is a solution to the following problem
\begin{equation}\label{Eucpb}
    \begin{cases}
    \sigma_k(D^2u)=C_n^k&\quad\text{in }\Omega,\\
    u=0&\quad\text{on }\partial\Omega,\\
    u_\gamma=c_0&\quad\text{on }\partial\Omega,
    \end{cases}
\end{equation}
with $k\in\{1,\cdots,n\}$ and $c_0$ a positive constant. Then upto a translation, $u=\frac{|x|^2-1}2$ and $\Omega$ is a ball of radius $c_0$.
\end{theorem}
In space forms, a few work has been done to generalize the Serrin's symmetry  to equation $\Delta u+nKu=c$ using the method of moving planes or P functions and Rellich-Pohozaev type identities (see \cite{Ciraolo2019,KP1998,Molzon1991,QX2017,Souam2005} and reference therein).

The hyperbolic space $\mathbb H^n$ can be described as the warped product space $[0,\infty)\times\mathbb S^{n-1}$ equipped with the rotationally symmetric metric 
\begin{equation}g=dr^2+h^2g_{\mathbb S^{n-1}},\end{equation} where $h=\sinh r$, $g_{\mathbb S^{n-1}}$ is the round metric on the $n-1$ dimensional sphere.

In the present paper, we consider the overdetermined problem below in hyperbolic space, 
\begin{equation}\label{mainpb}
    \begin{aligned}
        \begin{cases}
            \sigma_k(D^2u-uI)=C_n^k&\quad\text{in }\Omega,\\
            u=0&\quad\text{on }\partial \Omega,\\
            u_\gamma=c_0&\quad\text{on }\partial\Omega,
        \end{cases}
    \end{aligned}
\end{equation}
where $\Omega$ is a bounded $C^2$ domain of $\mathbb H^n$. Our result is the following: 
\begin{theorem}\label{mainthm}
Let $\Omega\subset \mathbb H^n$ be a $C^2$ bounded domain and $u\in C^3(\Omega)\cap C^2(\overline\Omega)$ be a solution to \eqref{mainpb} with $k\in\{1,2,\cdots,n\}$ and $c_0$ a positive constant. Then $\Omega$ is a geodesic ball $B_R$, and $u$ is radially symmetric.
\end{theorem}
By maximum principle, $u<0$ in $\Omega$, and the solution to Dirichlet problem of $\sigma_k(D^2u-uI)=C_n^k$ is unique. In theorem \ref{mainthm}, if we assume the center of $B_R$ is the origin, then 
$ u(r)=1-\frac{\cosh{r}}{\cosh{R}} $ is a unique solution to \eqref{mainpb}, 
where $r$ is the distance from $0$, $R$ and $c_0$ satisfy the relationship 
$ \frac{\sinh R}{\cosh R}=c_0 $.

In the next section, we recall some notations  concerning the Hessian equation in Euclidean space and hyperbolic space. Some known facts about curvature of level sets and Minkowskian integral formulas are also introduced there. In the third section, we use a P function and a Rellich-Pohozaev type identity for problem \eqref{Eucpb} to give an alternative proof of theorem \ref{2thm}. In the last section, we  derive the Rellich-Pohozaev type identity for problem \eqref{mainpb}, then combine with a P function to present the proof of theorem \ref{mainthm}.

\section{Preliminary}
\subsection{Elementary symmetric functions}
Let $\mathcal S^n$ be the space of real symmetric $n\times n$ matrices. We denote by $A=(a_{ij})$ a matrix in $\mathcal S^n$, and by $\lambda_1,\cdots, \lambda_n$ the eigenvalues of $A$. For $1\leq k\leq n$, we recall the definition of $k$-th elementary symmetric functions of $A$,

\begin{equation}\sigma_k(A)=\sigma_k(\lambda_1,\cdots,\lambda_n)=\sum_{1\leq i_1<\cdots<i_k\leq n}\lambda_{i_1}\cdots\lambda_{i_k}.\end{equation}
Denote by
\begin{equation}\sigma_k^{ij}(A):=\frac{\partial\sigma_k(A)}{\partial a_{ij}},\end{equation}
then it is easy to see from the definition above that
\begin{equation}\sigma_k(A)=\frac1k\sum_{i,j=1}^n\sigma_k^{ij}(A)a_{ij}\quad\text{and}\quad(n-k+1)\sigma_{k-1}(A)=\sum_{i=1}^n\sigma_k^{ii}(A).\end{equation}
For $1\leq k\leq n-1$, Newton's inequalities say that
\begin{equation}\label{Newton}(n-k+1)(k+1)\sigma_{k-1}(A)\sigma_{k+1}(A)\leq k(n-k)\sigma_k^2(A).\end{equation}
For $1\leq k\leq n$, recall that the Garding cone is defined as
\begin{equation}\Gamma_k=\{A\in\mathcal S^n: \sigma_1(A)>0,\cdots,  \sigma_k(A)>0\}.\end{equation}
Assume that $A\in \Gamma_k$, the following inequalities, known as MacLaurin inequalities hold,
\begin{equation}\label{MacLaurin}\frac{\sigma_k(A)}{C_n^k}\leq \frac{\sigma_l(A)}{C_n^l},\quad\forall\ k\geq l\geq 1.\end{equation}
The equalities hold if the eigenvalues $\lambda_1,\cdots,\lambda_{n}$ of $A$ are equal to each other.
Assume that $A\in \Gamma_k$, the following inequalities also holds,
\begin{equation}\label{NMieq}
    \frac{{\sigma_k(A)}/{C_n^k}}{{\sigma_{k-1}(A)}/{C_n^{k-1}}}\leq \frac{{\sigma_l(A)}/{C_n^l}}{{\sigma_{l-1}(A)}/{C_n^{l-1}}}\quad\forall \ k\geq l\geq 1.
\end{equation}
The following proposition can be found in \cite{Reilly1974}, see also \cite{Xia2021} for non-symmetric matrices.
\begin{proposition} \label{keyprop}
For any $n\times n$  matrix $A$, we have
\begin{equation}\label{prop2.1}
    \sigma_k^{ij}(A)=\sigma_{k-1}(A)\delta_{ij}-\sum_{l=1}^n\sigma_{k-1}^{il}(A)a_{jl}.
\end{equation}
\end{proposition}

In the following we write $D, D^2$ and $\Delta$ for the gradient, Hessian and Laplacian on hyperbolic space $\mathbb H^n$. For simplicity, we will use $u_i, \ u_{ij},\ \cdots $ and $u_\gamma$ to denote covariant derivatives and normal derivative of function $u$ with respect to the metric on $\mathbb H^n$. We write $X\cdot Y$ instead of $g(X,Y)$ for vector fields $X,Y$. We also follow Einstein's summation convention.

\subsection{Hessian operators}
\subsubsection{Hessian operators in Euclidean space}
Let $\Omega$ be an open subset of $\mathbb R^n$ and let $u\in C^2(\Omega)$. The $k$-Hessian operator $\mathcal S_k[u]$ is defined as the $k$-th elementary symmetric function $\sigma_k(D^2u)$ of $D^2u$. Notice that 
\begin{equation}\mathcal S_1[u]=\Delta u\quad\text{and}\quad \mathcal S_n[u]=\mathrm{det}D^2u.\end{equation}
A function $u$ is called $k$-convex in $\Omega$, if $D^2u(x)\in \Gamma_k$ for any $x\in \Omega$. A direct computation yields that $(\sigma_k^{1j}(D^2u),\cdots,\sigma_k^{nj}(D^2u))$ is divergence free, that is 
\begin{equation}\label{f2.10}
    \frac{\partial}{\partial x_i}\sigma_k^{ij}(D^2u)=0.
\end{equation}
By using \eqref{prop2.1}, it is easy to see that
\begin{equation}\label{f2.11}
    \sigma_k^{ij}(D^2u)u_{il}=\sigma_k^{il}(D^2u)u_{ij}.
\end{equation}
\subsubsection{Hessian operators in hyperbolic space}
Let $\Omega$ be an open subset of $\mathbb H^n$ and let $u\in C^2(\Omega)$. The $k$-Hessian operator $S_k[u]$ is defined as the $k$-th elementary symmetric function $\sigma_k(D^2u-uI)$ of $D^2u-uI$. Notice that 
\begin{equation}S_1[u]=\Delta u-nu\quad\text{and}\quad S_n[u]=\mathrm {det}(D^2u-uI).\end{equation}

A function $u$ is called \textit{$k$-admissible} in $\Omega$ if $D^2u(x)-u(x)I\in\Gamma_k$ for any $x\in \Omega$. A function $u\in C^2(\Omega)$ is called an \textit{admissible solution} of $S_k[u]=f$ in $\Omega$, if $u$ solves the equation and $u$ is an $k$-admissible function.

We list the following propositions with proofs, which will be used in the following sections.

\begin{proposition}\label{propdivfree}Suppose $u\in C^3(\Omega)$, then
\begin{equation}\label{divfree}
    D_i(\sigma_k^{ij}(D^2u-uI))=0.
\end{equation}
\end{proposition}

\begin{proof}
We prove it by induction. For $k=1$, \eqref{divfree} holds obviously. For $k=2$, since $u_{ilj}=u_{ijl}-u_l\delta_{ij}+u_j\delta_{il}$, we have
 \begin{equation}\begin{aligned}D_j\sigma_2^{ij}(D^2u-uI)=&D_j(\sigma_1(D^2u-uI)\delta_{ij}-\sigma_1^{js}(u_{si}-u\delta_{si}))\\=&(\Delta u-nu)_i-(u_{ji}-u\delta_{ji})_{j}=0.\end{aligned}\end{equation}
Suppose that \eqref{divfree} holds for $k-1$, then
\begin{equation}\begin{aligned}
 D_j(\sigma_k^{ij}(D^2u-uI))=&D_j(\sigma_{k-1}(D^2u-uI)\delta_{ij}-\sigma_{k-1}^{js}(D^2u-uI)(u_{si}-u\delta_{si}))\\
 =&D_i\sigma_{k-1}(D^2u-uI)-\sigma_{k-1}^{js}(D^2u-uI)(u_{si}-u\delta_{si})_{j}\\
 =&\sigma_{k-1}^{js}(D^2u-uI)((u_{js}-u\delta_{js})_i-(u_{si}-u\delta_{si})_{j})\\
 =&0.
 \end{aligned}\end{equation}
\end{proof}

\begin{proposition}
Let $u\in C^2(\Omega)$, then
\begin{equation}\label{changeindex}
    \sigma_k^{ij}(D^2u-uI)u_{il}=\sigma_k^{il}(D^2u-uI)u_{ij}.
\end{equation}
\end{proposition}

\begin{proof}
By proposition \ref{keyprop}, we obtain
\begin{equation}\begin{aligned}
\sigma_k^{ij}(D^2u-uI)(u_{il}-\delta_{il})
=&\big(\sigma_{k-1}(D^2u-uI)\delta_{ij}-\sigma_{k-1}^{is}(D^2u-uI)(u_{js}-u\delta_{js})\big)(u_{il}-\delta_{il})\\
=&\big(\sigma_{k-1}(D^2u-uI)\delta_{sl}-\sigma_{k-1}^{is}(D^2u-uI)(u_{il}-u\delta_{il})\big)(u_{js}-\delta_{js})\\
=&\sigma_{k}^{sl}(D^2u-uI)(u_{is}-u\delta_{js}).
\end{aligned}\end{equation}
Thus,
\begin{equation}\begin{aligned}
    \sigma_k^{ij}(D^2u-uI)u_{il}=&\sigma_k^{ij}(D^2u-uI)(u_{il}-u\delta_{il})+\sigma_k^{lj}(D^2u-uI)u\\=&\sigma_k^{il}(D^2u-uI)(u_{ij}-u\delta_{ij})+\sigma_k^{lj}(D^2u-uI)u\\=&\sigma_k^{il}(D^2u-uI)u_{ij}.
\end{aligned}\end{equation}
\end{proof}

\subsection{Minkowskian integral formulas}
Let $\Omega$ be a $C^2$ bounded domain, and $\partial\Omega$ is the boundary of $\Omega$. Denote the principle curvatures of $\partial\Omega$ by $\kappa=(\kappa_1,\cdots,\kappa_{n-1})$. The $k$-th curvature of $\partial\Omega$ is defined as
\begin{equation}H_k:=\sigma_k(\kappa)\quad k=1,\cdots,n-1.\end{equation}
$\Omega$ is called \textit{$k$-convex} with $k\in\{1,\cdots,n-1\}$, if $H_i>0$ for $i=1,\cdots,k$. In particular, $n-1$-convex is strictly convex, 1-convex is also called mean convex.

In the theory of convex bodies and differential geometry, Minkowskian integral formula (see\cite{Hsiung1954,Hsiung1956,Schneider1993}) is widely known. Suppose $\Omega$ is a domain of $\mathbb R^n$, then the Minkowskian integral formula says
\begin{equation}\label{Minkowskionrn}
    \int_{\partial\Omega}\frac{H_k}{C_{n-1}^{k}}x\cdot\gamma \mathrm d \sigma=\int_{\partial\Omega}\frac{H_{k-1}}{C_{n-1}^{k-1}} \mathrm d\sigma.
\end{equation}
Suppose $\Omega$ is a domain of $\mathbb H^n$. Let $p\in \Omega$, $r$ be the distance from $p$. Let $V(x)=\cosh(r(x))$, then $D^2V=VI$ and $\Delta V=nV$. Then the Minkowskian integral formula says
\begin{equation}\label{Minkowskionhn}
    \int_{\partial\Omega}\frac{H_k}{C_{n-1}^{k}}V_\gamma\mathrm d \sigma=\int_{\partial\Omega}\frac{H_{k-1}}{C_{n-1}^{k-1}}V \mathrm d\sigma.
\end{equation}

\subsection{Curvatures of level sets}
Let $u$ be a smooth function in space form $\mathbb R^n,\mathbb S^n$, or $\mathbb H^n$, for any regular $c\in\mathbb R$ of u (that is, $Du(x)\neq 0$ for any $x\in\mathbb H^n$ such that $u(x)=c$),
 the level set $\Sigma_c:=u^{-1}(c)$ is a smooth hypersurface by the implicit function theorem. The $k$-th order curvature $H_k$ of the level set $\Sigma_c$ is given by
\begin{equation}\label{curvoflvlset}
    H_{k-1}=\frac{\sigma_k^{ij}(D^2u)u_iu_j}{|Du|^{k+1}},
\end{equation}
which can be found in \cite{MZ2014,Reilly1974}.

\section{Overdetermined problem in $\mathbb R^n$}

In this section, we present a Rellich-Pohozaev type identity for $\sigma_k(D^2u)=C_n^k$ with zero Dirichlet boundary condition, and use a P function to give a proof of theorem \ref{2thm}.

The following lemma is proven in \cite{Salani2008}, which implies the solution to \eqref{Eucpb} is $k$-convex. This ensures MacLaurin inequalities \eqref{MacLaurin} can be applied.
\begin{lemma}[\cite{Salani2008}]\label{ukconvex}
Let $\Omega\subset\mathbb R^n$ be a bounded $C^2$ domain and $u\in C^2(\overline\Omega)$ is a solution to \eqref{Eucpb}, then $u$ is $k$-convex in $\Omega$.
\end{lemma}
P functions have been extensively investigated and inspired several effective works in the context of elliptic partial differential equations. For fully nonlinear equations, the P function for 2-dimensional Monge-Amp\`ere equation was given by Ma \cite{Ma1999}, and the P functions for $k$-Hessian equations and $k$-curvature equations were given by Philippin and Safoui \cite{Safoui}. Based on lemma \ref{ukconvex}, we are able to 
derive the lemma below, which was prove in \cite{Safoui}. So we can apply the maximum principle on the P function. For completeness, we present the proof. 
\begin{lemma}[\cite{Safoui}]\label{P}
Let $u\in C^3(\Omega)$ be an admissible solution of $\sigma_k(D^2u)=C_n^k$ in $\Omega\subset\mathbb R^n$. Then the following P function 
\begin{equation}\label{eq:p}P:=|Du|^2-2u\end{equation} satisfies 
\begin{equation}\sigma_k^{ij}(D^2u)P_{ij}\geq 0.\end{equation}
\end{lemma}
\begin{proof}
Direct computation leads that
\begin{equation}\label{ineq1}\begin{aligned}\sigma_k^{ij}(D^2u)P_{ij}=&2\sigma_k^{ij}(D^2u)\big(u_{li}u_{lj}+u_lu_{lij}-u_{ij})=2\sigma_k^{ij}(D^2u)u_{li}u_{lj}-2k\sigma_k(D^2u)\\=&2\big(\sigma_1(D^2 u)\sigma_k(D^2u)-(k+1)\sigma_{k+1}(D^2u)-k\sigma_k(D^2u)\big)\end{aligned}\end{equation}
If $\sigma_{k+1}(D^2u)> 0$, then $D^2u\in\Gamma_{k+1}$. By  \eqref{NMieq}, we have 
\begin{equation}\label{ls1}(k+1)\sigma_{k+1}(D^2u)\leq \frac{n-k}n\sigma_1(D^2u)\sigma_k(D^2u).\end{equation}
If $\sigma_{k+1}\leq 0$, \eqref{ls1} holds naturally. So by using \eqref{MacLaurin}, we get
\begin{equation}\label{ineq1-1}\sigma_k^{ij}(D^2u)P_{ij}\geq 2k\sigma_k(\frac{\sigma_1(D^2u)}{n}-\frac{\sigma_k(D^2u)}{C_n^k})\geq 0.\end{equation}
\end{proof}

The Rellich-Pohozaev type identity for Hessian equation has already been found \cite{Tso1990, Salani2008}. 
Brandolini, Nitsch, Salani and Trombetti \cite{Salani2008} gave the Rellich-Pohozaev type identity for $\sigma_k(D^2u)=f(u)$ in $\Omega$, with $u=0$ on $\partial\Omega$,
\begin{equation}\frac{n-2k}{k(k+1)}\int_\Omega\sigma_k^{ij}(D^2u)u_iu_j\mathrm dx+\frac1{k+1}\int_{\partial\Omega}x\cdot\gamma|Du|^{k+1}H_{k-1}\mathrm d\sigma=n\int_\Omega F(u)\mathrm dx,\end{equation} where $F(u)=\int_u^0f(s)\mathrm ds$.
In their proof, they need to deal with the fourth order term $\frac{\partial^2}{\partial x_i\partial x_l}\sigma_k^{ij}(D^2u)$, which may be difficult in space forms. To overcome it, we come up with a different Rellich-Pohozaev type identity.
\begin{lemma}\label{dlempohorn}
Let $\Omega\subset \mathbb R^n $ be a bounded $C^2$ domain. If $u\in C^2(\Omega)\cap C^1(\overline\Omega)$ is a solution to the problem
\begin{equation}\label{f3.6}\begin{cases}
    \sigma_k(D^2u)=C_n^k &\quad\text{in }\Omega,\\
    u=0 &\quad\text{on }\partial\Omega.
\end{cases}\end{equation}
Then 
\begin{equation}\label{PohoEuc}
\frac12\int_{\partial\Omega}\sigma_k^{ij}(D^2u)x_j\gamma_i|Du|^2\mathrm d\sigma=-kC_n^k\int_\Omega u\mathrm dx+\frac{n-k+1}2\int_\Omega\sigma_{k-1}(D^2u)|Du|^2\mathrm dx.\end{equation}
\end{lemma}

\begin{proof}
Multiplying the equation \eqref{f3.6} by $u$, we obtain
\begin{equation}\label{f3.8}
    kC_n^ku=\sigma_k^{ij}(D^2u)u_{ij}u=\frac12\sigma_k^{ij}(D^2u)u_{il}u(|x|^2)_{lj}.
\end{equation}
By \eqref{f2.10}, we have
\begin{equation}\label{f3.9}
    \sigma_k^{ij}(D^2u)u_{il}u(|x|^2)_{lj}=(\sigma_k^{ij}(D^2u)u_{il}u(|x|^2)_l)_j-\sigma_k^{ij}(D^2u)u_{ilj}u(|x|^2)_l-\sigma_k^{ij}(D^2u)u_{il}u_j(|x|^2)_l.
\end{equation}
It follows from \eqref{f3.6} that
\begin{equation}\label{f3.10}
    \sigma_k^{ij}(D^2u)u_{ijl}=0.
\end{equation}
By \eqref{f2.11}, we have
\begin{equation}\label{f3.11}
    \sigma_k^{ij}(D^2u)u_{il}u_j(|x|^2)_l=2\sigma_k^{il}(D^2u)u_{ij}u_jx_l=\sigma_k^{ij}(D^2u)(|Du|^2)_ix_j.
\end{equation}
Putting \eqref{f3.8}, \eqref{f3.9}, \eqref{f3.10} and \eqref{f3.11} together, and integrating it on $\Omega$, we find
\begin{equation}
    kC_n^k\int_\Omega u\mathrm dx=\frac12\int_\Omega (\sigma_k^{ij}(D^2u)u_{il}u(|x|^2)_l)_j\mathrm dx-\frac12\int_\Omega \sigma_k^{ij}(D^2u)(|Du|^2)_ix_j\mathrm dx.
\end{equation}
It follows from divergence theorem and $u=0$ on $\partial\Omega$ that
\begin{equation}\label{f3.13}
    kC_n^k\int_\Omega u\mathrm dx=-\frac12\int_\Omega \sigma_k^{ij}(D^2u)(|Du|^2)_ix_j\mathrm dx.
\end{equation}
Using \eqref{f2.10} again, we get
\begin{equation}\label{f3.14}
    \sigma_k^{ij}(D^2u)(|Du|^2)_ix_j=(\sigma_k^{ij}(D^2u)|Du|^2x_j)_i-(n-k+1)\sigma_{k-1}(D^2u)|Du|^2.
\end{equation}
Substituting \eqref{f3.14} into \eqref{f3.13} and using divergence theorem again, we finally get \eqref{PohoEuc}.
\end{proof}
The following lemma help us to deal with the term of boundary integral in \eqref{PohoEuc}.
\begin{lemma}\label{dlemma3.3}
Let $\Omega\subset \mathbb R^n$ be a $C^2$ bounded domain, $u\in C^2(\overline\Omega)$ be a solution to the equation \eqref{Eucpb}, then 
\begin{equation}\label{formulaforlhs}
   \sigma_k^{ij}(D^2u)x_j\gamma_i|Du|^2 =\sigma_k^{ij}(D^2u)u_iu_jx\cdot\gamma\quad\text{on }\partial\Omega.
\end{equation}
\end{lemma}
\begin{proof}
By the boundary conditions of problem \eqref{Eucpb}, we have on $\partial\Omega$ that
\begin{equation}\label{Duonpo}
    Du=u_\gamma\gamma,
\end{equation}
and 
\begin{equation}
    \begin{aligned}
    (u_\gamma)_j=(u_i\gamma_i)_j=u_{ij}\gamma_i+u_iD_j\gamma_i=u_{ij}\gamma_i+u_\gamma \gamma_iD_j\gamma_i=u_{ij}\gamma_i.
    \end{aligned}
\end{equation}
Then 
\begin{equation}
    \begin{aligned}
    (u_\gamma)_\gamma=u_{ij}\gamma_i\gamma_j:=u_{\gamma\gamma}.
    \end{aligned}
\end{equation}
Since $u_\gamma=c_0$ on $\partial\Omega$, for any tangential direction $\tau$, we have $(u_\gamma)_\tau=0$, then
\begin{equation}
    \begin{aligned}
    Du_\gamma=(u_\gamma)_\gamma\gamma=u_{ij}\gamma_i\gamma_j\gamma.
    \end{aligned}
\end{equation}
Therefore, 
\begin{equation}\label{3.10}
    \begin{aligned}
    u_{ij}\gamma_i=(u_\gamma)_j=u_{kl}\gamma_k\gamma_l\gamma_j=(u_{\gamma})_{\gamma}\gamma_j=u_{\gamma\gamma}\gamma_j.
    \end{aligned}
\end{equation}
It follows that 
\begin{equation}\label{3.11}u_{ij}u_i=u_{\gamma\gamma}u_j.\end{equation}
Now we derive \eqref{formulaforlhs} by induction. When $k=1$, \eqref{formulaforlhs} holds obviously since $\sigma_1^{ij}(D^2u)=\delta_{ij}$. Suppose \eqref{formulaforlhs} holds for $k-1$, then 
\begin{equation}\label{e3.15}\begin{aligned}\sigma_k^{ij}(D^2u)x_i\gamma_j|Du|^2=&(\sigma_{k-1}(D^2u)\delta_{ij}-\sigma_{k-1}^{js}(D^2u)u_{si})x_j\gamma_i|Du|^2\\
=&\sigma_{k-1}(D^2u)|Du|^2x\cdot\gamma-\sigma_{k-1}^{js}(D^2u)u_{\gamma\gamma}\gamma_sx_j|Du|^2\\
=&\sigma_{k-1}(D^2u)|Du|^2x\cdot\gamma-\sigma_{k-1}^{ij}(D^2u)u_iu_jx\cdot\gamma u_{\gamma\gamma}\\
=&(\sigma_{k-1}(D^2u)|Du|^2-\sigma_{k-1}^{ij}(D^2u)u_{il}u_lu_j)x\cdot\gamma\\
=&\sigma_{k}^{ij}(D^2u)u_iu_jx\cdot\gamma,
\end{aligned}\end{equation}
where we use \eqref{prop2.1} in the first and last equality, \eqref{3.10} in the second equality, and \eqref{3.11} in the fourth equality.
\end{proof}
Moreover, by formulas for the curvature of level sets and Minkowskian integral formulas, the boundary integral can be turned into an integral on $\Omega$.
\begin{lemma}\label{dlemma3.4}
Let $u\in C^3(\Omega)\cap C^2(\overline\Omega)$ be a solution to the problem \eqref{Eucpb}, then 
\begin{equation}\label{d3.13}
    \int_{\partial\Omega}\sigma_k^{ij}(D^2u)u_iu_jx\cdot\gamma\mathrm d\sigma=(n-k+1)c_0^2\int_{\Omega}\sigma_{k-1}(D^2u)\mathrm dx.
\end{equation}
\end{lemma}
\begin{proof}
From \eqref{curvoflvlset}, 
\begin{equation}
    \sigma_k^{ij}u_iu_jx\cdot \gamma=H_{k-1}|Du|^{k+1}x\cdot \gamma\quad\text{on }\partial\Omega,
\end{equation}
Integrating it on $\partial\Omega$, and using Minkowskian integral formula \eqref{Minkowskionrn}, we obtain,
\begin{equation}\label{d3.15}
    \int_{\partial\Omega}\sigma_k^{ij}(D^2u)u_iu_jx\cdot \gamma\mathrm d\sigma=\int_{\partial\Omega}H_{k-1}|Du|^{k+1}x\cdot \gamma\mathrm d\sigma =\frac{n-k+1}{k-1}\int_{\partial\Omega}H_{k-2}|Du|^{k+1}\mathrm d\sigma.
\end{equation}
By \eqref{curvoflvlset} and \eqref{Duonpo}, we have
\begin{equation}
    H_{k-2}|Du|^{k+1}=\sigma_{k-1}^{ij}(D^2u)u_iu_j|Du|=c_0^2\sigma_{k-1}^{ij}(D^2u)u_i\gamma_j\quad\text{on }\partial\Omega.
\end{equation}
Applying the divergence theorem, we get
\begin{equation}\label{d3.17}
\begin{aligned}
\int_{\partial\Omega}H_{k-2}|Du|^{k+1}\mathrm  d\sigma=&c_0^2\int_{\Omega}(\sigma_{k-1}^{ij}(D^2u)u_i)_j\mathrm dx
\end{aligned}
\end{equation}
Since $\sigma_k^{ij}(D^2u)$ is divergence free, it follows that
\begin{equation}\label{d3.18}
    (\sigma_{k-1}^{ij}(D^2u)u_i)_j=\sigma_{k-1}^{ij}(D^2u)u_{ij}=(k-1)\sigma_{k-1}(D^2u).
\end{equation}
Putting \eqref{d3.15}, \eqref{d3.17} and \eqref{d3.18} together, we deduce \eqref{d3.13}.
\end{proof}

\begin{proof}[Proof of Theorem \ref{2thm}.]
By lemma \ref{dlempohorn}, lemma \ref{dlemma3.3} and lemma \ref{dlemma3.4}, we obtain
\begin{equation}\label{test2}\frac{n-k+1}2\int_{\Omega}\sigma_{k-1}(D^2 u)(|Du|^2-c_0^2) \mathrm dx =kC_n^k\int_\Omega u \mathrm dx .\end{equation}
By maximum principle and lemma \ref{P}, we have
\begin{equation}
    P=|Du|^2-2u\leq c_0^2\quad\text{in }\Omega,
\end{equation}
Substituting it into \eqref{test2}, we get
\begin{equation}\label{g3.32}
    kC_n^k\int_\Omega u\mathrm dx\leq (n-k+1)\int_\Omega \sigma_{k-1}(D^2u)u\mathrm dx.
\end{equation}
On the other hand, by MacLaurin inequalities \eqref{MacLaurin},  $\sigma_{k-1}(D^2u)\geq C_n^{k-1}(\frac{\sigma_k(D^2u)}{C_n^k})^\frac{k-1}k=C_n^{k-1}$ in $\Omega$. Since $u<0$ in $\Omega$, we have
\begin{equation}\label{g3.33}
    (n-k+1)\sigma_{k-1}(D^2u)u\leq kC_n^ku.
\end{equation}
It follows from \eqref{g3.32} and \eqref{g3.33} that
\begin{equation}
    \sigma_{k-1}(D^2u)=C_n^{k-1}\quad\text{in }\Omega.
\end{equation}
By MacLaurin inequalities \eqref{MacLaurin},  eigenvalues of $D^2u$ are all equal to 1. Then $u=\frac{|x-x_0|^2-c_0^2}2$ and $\Omega=B_{c_0}(x_0)$ for some $x_0$. Hence we complete the proof of theorem \ref{2thm}
\end{proof}

\section{Overdetermined problem in $\mathbb H^n$}
In this section, we present a Rellich-Pohozaev type identity for $\sigma_k(D^2u-uI)=C_n^k$ with zero Dirichlet boundary condition, and establish a differential inequality for a P function $P=|Du|^2-u^2-2u$. After a similar argument as in the third section, we give a proof of theorem \ref{mainthm}.

Paralleled with lemma \ref{ukconvex}, we prove the following lemma.

\begin{lemma}\label{dlemma4.5}
Let $\Omega\subset\mathbb H^n$ be a bounded $C^2$ domain and $u\in C^2(\overline\Omega)$ is a solution to the problem \eqref{mainpb}, then $u$ is $k$-admissibe in $\Omega$.
\end{lemma}

\begin{proof}
The proof is almost the same as the proof in \cite{Salani2008}, where they prove that $u$ is a k-convex function if it is a solution to the problem \eqref{Eucpb}. From the boundedness and the smoothness of $\Omega$, there exists a point $x_0\in \partial\Omega$, the principal curvatures $\kappa_1,\cdots,\kappa_{n-1}$ of $\partial\Omega$ at $x_0$ are nonnegative. By choosing a suitable coordinate system centered at $x_0$, the first $n-1$ axes lay in the principal directions of curvature and the last one points in the direction of outer normal $\gamma_{x_0}$ of $\partial\Omega$ at $x_0$. Then $D^2u$ has the following form:
\begin{equation}D^2u=\begin{aligned}\begin{pmatrix}c_0\kappa_1&\cdots &0&u_{1n}\\\vdots&\ddots&\vdots&\vdots\\ 0&\cdots&c_0\kappa_{n-1}&u_{n-1\, n}\\u_{n1}&\cdots&u_{n\, n-1}&u_{nn} \end{pmatrix}\end{aligned}\end{equation}
Using the boundary condition in equation \eqref{mainpb}, a straightforward calculation shows that $D^2u(x_0)$ is in fact in diagonal form. Hence,
\begin{equation}0<S_k[u](x_0)=u_{nn}(x_0)c_0^{k-1}H_{k-1}(x_0)+c_0^k H_k(x_0),\end{equation}
Then from Newton inequalities \eqref{Newton}, we deduce that
\begin{equation}\begin{aligned}S_{k-1}[u](x_0)=&u_{nn}(x_0)c_0^{k-2}H_{k-2}(x_0)+c_0^{k-1}H_{k-1}(x_0)\\>&-c_0^{k-1}\frac{H_{k-2}(x_0)H_{k}(x_0)}{H_{k-1}(x_0)}+c_0^{k-1}H_{k-1}(x_0)\geq 0.\end{aligned}\end{equation}
Repeat the process untill we get $S_j[u](x_0)>0$ for all $1\leq j\leq k$. Hence $D^2u(x_0)-u(x_0)I\in \Gamma_k$. From the smoothness of $u$, $u$ is $k$-admissible in $\Omega$.
\end{proof}

 Distinct from the Euclidean case, the extra term $-uI$ in equation \eqref{mainpb} prompts us to consider the $P$ function $P=|Du|^2-u^2-2u$, as  in \cite{Ciraolo2019,QX2017}.
\begin{lemma}\label{PP}
Let $u\in C^3(\Omega)$ be an admissible solution of $\sigma_k(D^2u-uI)=C_n^k$ in $\Omega\subset\mathbb H^n$. Then the following P function 
\begin{equation}\label{eq:pp}\widetilde P:=|Du|^2-u^2-2u\end{equation} satisfies 
\begin{equation}\sigma_k^{ij}(D^2u-uI)\widetilde P_{ij}\geq 0.\end{equation}
\end{lemma}

\begin{proof}
We may choose a suitable coordinate such that $g_{ij}=\delta_{ij}$ and $g_{ij,k}=0$ at $p\in \Omega$. The following computation is done at $p$.
Since\begin{equation}\begin{aligned}\frac12(|Du|^2)_{ij}=&u_{li}u_{lj}+u_lu_{lij}=u_{li}u_{lj}+u_lu_{ijl}-u_lu_p(\delta_{pl}\delta_{ij}-\delta_{pj}\delta_{il})\\
=&(u_{li}-u\delta_{li})(u_{lj}-u\delta_{lj})+2u(u_{ij}-u\delta_{ij})+u^2\delta_{ij}+u_l(u_{ij}-u\delta_{ij})_{l}+u_iu_j,\end{aligned}\end{equation}
and
\begin{equation}\begin{aligned}\frac12(u^2)_{ij}=&uu_{ij}+u_iu_j=u(u_{ij}-u\delta_{ij})+u^2\delta_{ij}+u_iu_j.\end{aligned}\end{equation}
We have
\begin{equation}\label{ineq2}\begin{aligned}\frac12\sigma_k^{ij}(D^2u-uI)\widetilde P_{ij}=&\sigma_k^{ij}(D^2u-uI)((u_{li}-u\delta_{li})(u_{lj}-u\delta_{lj})+u(u_{ij}-u\delta_{ij})\\&+u_l(u_{ij}-u\delta_{ij})_{l}-(u_{ij}-u\delta_{ij})-u\delta_{ij})\\
=&(S_1[u]S_k[u]-(k+1)S_{k+1}[u]-kS_k[u])\\&-u(-kS_k[u]+(n-k+1)S_{k-1}[u])\\\geq &0,\end{aligned}\end{equation}
where the first term in the third line is dealt as in lemma \ref{P} and the second term can be handle with by MacLaurin inequalities \eqref{MacLaurin}, since $u$ is negative in $\Omega$.
\end{proof}
Now we establish the following Rellich-Pohozaev type identity.
\begin{lemma}\label{dlempohoonhn}
Let $\Omega\subset\mathbb H^n$ be a bounded $C^2$ domain. Let $u\in C^3(\Omega)\cap C^2(\overline\Omega)$ be a solution to the problem  
\begin{equation}\begin{cases}\label{eq4.9}
\sigma_k(D^2u-uI)=C_n^k&\quad\text{in }\Omega,\\
u=0&\quad\text{on }\partial\Omega.
\end{cases}\end{equation}
Then 
\begin{equation}\label{pohozaev}
    \begin{aligned}
    \frac12\int_{\partial\Omega}\sigma_k^{ij}(D^2u-uI) V_i\gamma_j|Du|^2\mathrm d\sigma
    =&-\frac{n-k+1}2\int_\Omega \sigma_{k-1}(D^2u-uI)u^2V\mathrm dx-kC_n^k\int_\Omega uV\mathrm dx\\
    &+\frac{n-k+1}2\int_\Omega \sigma_{k-1}(D^2u-uI)|Du|^2V\mathrm dx.
    \end{aligned}
\end{equation}
\end{lemma}
\begin{proof}
Multiplying the equation\eqref{eq4.9} by $uV$, we obtain 
\begin{equation}\label{c4.4}
\begin{aligned}
    kC_n^kuV=&\sigma_k^{ij}(D^2u-uI)(u_{ij}-u\delta_{ij})uV\\
    =&\sigma_k^{ij}(D^2u-uI)u_{ij}uV-(n-k+1)\sigma_{k-1}(D^2u-uI)u^2V\\
    :=&I+II.
\end{aligned}
\end{equation}
Since $D^2V=VI$, by \eqref{divfree}, we have
\begin{equation}
    \begin{aligned}
    I:=&\sigma_k^{ij}(D^2u-uI)u_{ij}uV
    =\sigma_k^{ij}(D^2u-uI)u_{il}uV_{lj}\\
    =&(\sigma_k^{ij}(D^2u-uI)u_{il}uV_l)_j-\sigma_k^{ij}(D^2u-uI)u_{ilj}uV_l-\sigma_k^{ij}(D^2u-uI)u_{il}u_jV_l
    \end{aligned}
\end{equation}
Using 
$u_{ilj}=u_{ijl}-u_l\delta_{ij}+u_j\delta_{ij}$, it follows that
\begin{equation}\label{c4.6}
\begin{aligned}
    I=&(\sigma_k^{ij}(D^2u-uI)u_{il}uV_l)_j-\sigma_k^{ij}(D^2u-uI)(u_{ij}-u\delta_{ij})_luV_l\\
    &-\sigma_k^{ij}(D^2u-uI)u_juV_i-\sigma_k^{ij}(D^2u-uI)u_{il}u_jV_l.
\end{aligned}
\end{equation}
Differentiating the equation $\sigma_k(D^2u-uI)=C_n^k$, we get
\begin{equation}\label{c4}\sigma_k^{ij}(D^2u-uI)(u_{ij}-u\delta_{ij})_l=0.\end{equation}
By \eqref{divfree}, we compute that
\begin{equation}\label{c4.7}
    \begin{aligned}
    \sigma_k^{ij}(D^2u-uI)u_juV_i=&\frac12\sigma_k^{ij}(D^2u-uI)(u^2)_jV_i\\
    =&\frac12(\sigma_k^{ij}(D^2u-uI)u^2V_i)_j-\frac12\sigma_k^{ij}(D^2u-uI)u^2V_{ij}\\
    =&\frac12(\sigma_k^{ij}(D^2u-uI)u^2V_i)_j-\frac{n-k+1}2\sigma_{k-1}(D^2u-uI)u^2V.
    \end{aligned}
\end{equation}
By \eqref{divfree} and \eqref{changeindex}, we also have that
\begin{equation}\label{c4.8}
\begin{aligned}
\sigma_k^{ij}(D^2u-uI)u_{il}u_jV_l=&\sigma_k^{il}(D^2u-uI)u_{ij}u_jV_l
=\frac12\sigma_k^{ij}(D^2u-uI)(|Du|^2)_iV_j\\
=&\frac12(\sigma_k^{ij}(D^2u-uI)|Du|^2V_j)_i-\frac{n-k+1}2\sigma_{k-1}(D^2u-uI)|Du|^2V
\end{aligned}
\end{equation}
Subtituting  \eqref{c4.6}, \eqref{c4}, \eqref{c4.7} and \eqref{c4.8} into \eqref{c4.4}, we obtain
\begin{equation}
    \begin{aligned}
    kC_n^kuV=&(\sigma_k^{ij}(D^2u-uI)u_{il}uV_l)_j-(n-k+1)\sigma_{k-1}(D^2u-uI)u^2V\\
    &-\frac12(\sigma_k^{ij}(D^2u-uI)u^2V_i)_j+\frac{n-k+1}2\sigma_{k-1}(D^2u-uI)u^2V\\
    &-\frac12(\sigma_k^{ij}(D^2u-uI)|Du|^2V_j)_i+\frac{n-k+1}2\sigma_{k-1}(D^2u-uI)|Du|^2V.
    \end{aligned}
\end{equation}
Applying divergence theorem and noting that $u=0$ on $\partial\Omega$, we deduce 
\begin{equation}
    \begin{aligned}
    kC_n^k\int_{\Omega}uV\mathrm dx
    =&\frac{n-k+1}2\int_\Omega \sigma_{k-1}(D^2u-uI)|Du|^2 V\mathrm dx-\frac{n-k+1}2\int_\Omega\sigma_{k-1}(D^2u-uI)u^2V\mathrm dx\\
    &-\frac12\int_{\partial\Omega}\sigma_k^{ij}(D^2u-uI)|Du|^2V_j\gamma_i\mathrm d\sigma.
    \end{aligned}
\end{equation}
\end{proof}
To handle with the boundary integral term, we need the following lemma. We omit the proof which is the same as lemma \ref{dlemma3.3}.
\begin{lemma}\label{dlemma4.3}
Let $\Omega\subset \mathbb H^n$ be a $C^2$ bounded domain, $u\in C^3(\overline\Omega)$ be a solution to problem \eqref{mainpb}, then 
\begin{equation}
   \sigma_k^{ij}(D^2u-uI)V_j\gamma_i|Du|^2 =\sigma_k^{ij}(D^2u-uI)u_iu_jV_\gamma\quad\text{on }\partial\Omega.
\end{equation}
\end{lemma}
Furthermore, we turn the boundary integral into an integral on $\Omega$.
\begin{lemma}\label{dlemma4.4}
Let $u\in C^3(\Omega)\cap C^2(\overline\Omega)$ be a solution to problem \eqref{mainpb}, then 
\begin{equation}\label{d4.13}
    \int_{\partial\Omega}\sigma_{k}^{ij}(D^2u-uI)u_iu_jV_\gamma\mathrm d\sigma=(n-k+1)c_0^2\int_\Omega \sigma_{k-1}(D^2u-uI)V\mathrm dx.
\end{equation}
\end{lemma}

\begin{proof}
Using  Minkowskian integral formula \eqref{Minkowskionrn} and \eqref{curvoflvlset}, we obtain
\begin{equation}\begin{aligned}\label{d4.15}
    \int_{\partial\Omega}\sigma_k^{ij}(D^2u)u_iu_jV_\gamma\mathrm d\sigma
    =&\int_{\partial\Omega}H_{k-1}|Du|^{k+1}V_\gamma\mathrm d\sigma
    =\frac{n-k+1}{k-1}\int_{\partial\Omega}H_{k-2}|Du|^{k+1}V\mathrm d\sigma\\
    =&\frac{n-k+1}{k-1}\int_{\partial\Omega}\sigma_{k-1}^{ij}(D^2u)u_iu_j|Du|V\mathrm d\sigma.
\end{aligned}\end{equation}
Since $u=0$ and $|Du|=u_\gamma=c_0$ on $\partial\Omega$, we have
\begin{equation}\label{d4.17}
\begin{aligned}
\int_{\partial\Omega}\sigma_{k-1}^{ij}(D^2u)u_iu_j|Du|V\mathrm d\sigma
=&c_0^2\int_{\partial\Omega}\sigma_{k-1}^{ij}(D^2u-uI)u_iV\gamma_j\mathrm dx\\
=&c_0^2\int_{\Omega}(\sigma_{k-1}^{ij}(D^2u-uI)u_iV)_j\mathrm dx
\end{aligned}
\end{equation}

Applying proposition \ref{propdivfree}, we get
\begin{equation}\label{d4.18}
    (\sigma_{k-1}^{ij}(D^2u-uI)u_iV)_j=\sigma_{k-1}^{ij}(D^2u-uI)u_{ij}V+\sigma_{k-1}^{ij}(D^2u-uI)u_iV_j.
\end{equation}
By computation, 
\begin{equation}\label{d4.19}
    \sigma_{k-1}^{ij}(D^2u-uI)u_iV_j=(\sigma_{k-1}^{ij}(D^2u-uI)uV_j)_i-\sigma_{k-1}^{ij}(D^2u-uI)uV\delta_{ij}.
\end{equation}
Substituting \eqref{d4.18} and \eqref{d4.19} into \eqref{d4.17}, and noting that $u=0$ on $\partial\Omega$, it follows that
\begin{equation}\label{d4.20}\begin{aligned}
\int_{\partial\Omega}\sigma_{k-1}^{ij}(D^2u)u_iu_j|Du|V\mathrm d\sigma
=&c_0^2\int_{\Omega}\sigma_{k-1}^{ij}((D^2u-uI)(u_{ij}-u\delta_{ij})V dx\\
=&(k-1)c_0^2\int_{\Omega}\sigma_{k-1}(D^2u-uI)V\mathrm dx.
\end{aligned}
\end{equation}
Putting \eqref{d4.15} and \eqref{d4.20} together, we obtain \eqref{d4.13}.
\end{proof}

\begin{proof}[Proof of theorem \ref{mainthm}]
Combining lemma \ref{dlempohoonhn}, lemma \ref{dlemma4.3} with lemma \ref{dlemma4.4}, we obtian
\begin{equation}\label{test1}\begin{aligned}kC_n^k\int_\Omega uV \mathrm dx =&\frac{n-k+1}2\int_\Omega \sigma_{k-1}(D^2u-uI)V(|Du|^2-u^2-c_0^2) \mathrm dx .\end{aligned}\end{equation}
By lemma \ref{PP}, maximum principle can be applied to $\widetilde P$, thus
\begin{equation}
    \widetilde P=|Du|^2-u^2-2u\leq c_0^2\quad\text{in }\Omega,
\end{equation}
Putting it into \eqref{test1}, we deduce
\begin{equation}\label{g4.30}
    kC_n^k\int_\Omega u\mathrm dx\leq (n-k+1)\int_\Omega \sigma_{k-1}(D^2u-uI)u\mathrm dx.
\end{equation}
On the other hand, using MacLaurin inequalities \eqref{MacLaurin}, 
 we have 
 \begin{equation}
 \sigma_{k-1}(D^2u-uI)\geq C_n^{k-1}(\frac{\sigma_k(D^2u-uI)}{C_n^k})^\frac{k-1}k=C_n^{k-1}.
 \end{equation}
 Since $u<0$ in $\Omega$, we get 
\begin{equation}\label{g4.31}
    (n-k+1)\sigma_{k-1}(D^2u-uI)u\leq kC_n^ku\quad \text{in }\Omega.
\end{equation}
It follows from \eqref{g4.30} and \eqref{g4.31} that
\begin{equation}
    \sigma_{k-1}(D^2u-uI)=C_n^{k-1}\quad\text{in }\Omega.
\end{equation}
By MacLaurin inequalities \eqref{MacLaurin}, eigenvalues of $D^2u-uI$ are all equal to 1. Follows from an Obata type result (\cite{Reilly1980}, See also \cite{Catino2012,CheegerColding1996,Ciraolo2019}), $\Omega$ must be a ball $B_R$ and $u$ depends only on the distance from the center of $B_R$, where $R=\tanh^{-1}c_0$. It is easy to see that $u$ is of the form 
\begin{equation}u=1-\frac{\cosh r}{\cosh R}.\end{equation}
Hence we complete the proof of theorem \ref{mainthm}.
\end{proof}

\section*{Acknowledgments}
The authors would like to thank Prof. Xi-Nan Ma for his discussions and advice. The research is supported by the National Science Foundation of China No. 11721101 and the National Key R and D Program of China 2020YFA0713100. 



\begin{thebibliography}{99}
\bibitem{Salani2008}
B. Brandolini, C.  Nitsch, P. Salani and  C. Trombetti, Serrin-type overdetermined problems: An alternative proof.
\textit{Arch. Ration. Mech. Anal.}\textbf{\ 190}(2)\ 267--280(2008).
\bibitem{Catino2012}
G. Catino, C. Mantegazza and L. Mazzieri, On the global structure of conformal gradient solitons with nonnegative Ricci tensor.
{Commun. Contemp. Math.}{\ 14}(6)\ (2012).
\bibitem{CheegerColding1996}
J. Cheeger and T. H. Colding, Lower bounds on Ricci curvature and the almost rigidity of warped products.
{Ann. Math.}{\ 144}(1)\ 189--237(1996).
\bibitem{Ciraolo2019}
G. Ciraolo and L. Vezzoni, On Serrin's overdetermined problem in space forms.
{Manuscripta Math.}{\ 159}(3-4)\ 445--452(2019).
\bibitem{Xia2021}
F. Della Pietra, N. Gavitone and C. Xia, Symmetrization with respect to mixed volumes.
{Adv. Math.}{\ 388}(2021).
\bibitem{FK2008}
A. Farina and  B. Kawohl, Remarks on an overdetermined boundary value problem.
{Calc. Var. Partial Differ. Equ.}{\ 31}(3)\ 351--357(2008).
\bibitem{FGK2006}
I. Fragala, F. Gazzola and B. Kawohl, Overdetermined problems with possibly degenerate ellipticity, a geometric approach.
{Math. Z.}{\ 254}(1)\ 117--132(2006).
\bibitem{Garofalo1989}
N. Garofalo and  J. L. Lewis, A symmetry result related to some overdetermined boundary-value problems.
{Am. J. Math.}{\ 111}(1)\ 9--33(1989).
\bibitem{Hsiung1954}
C.-C. Hsiung, Some integral formulas for closed hypersurfaces. Math. Scand.
2, 286-294 (1954)
\bibitem{Hsiung1956}
C.-C. Hsiung, Some integral formulas for closed hypersurfaces in Riemannian space. {Pacific J. Math.}{\ 6} 291--299(1956).
\bibitem{Jia2020}
X. H. Jia, Overdetermined problems for Weingarten hypersurfaces.
{Calc. Var. Partial Differ. Equ.}{\ 59}(2)\ 2020).
\bibitem{KP1998}
S. Kumaresan and  J. Prajapat, Serrin's result for hyperbolic space and sphere.
{Duke Math. J.}{\ 91}(1)\ 17--28(1998).
\bibitem{Ma1999}
X. N. Ma, A necessary condition of solvability for the capillarity boundary of Monge-Ampere equations in two dimensions.
{Proc. Amer. Math. Soc.}{\ 127}(3)\ 763--769(1999).
\bibitem{MZ2014}
X. N. Ma and  Y. B. Zhang, The convexity and the Gaussian curvature estimates for the level sets of harmonic functions on convex rings in space forms.
{J. Geom. Anal.}{\ 24}(1)\ 337--374(2014).
\bibitem{Molzon1991}
R. Molzon,  Symmetry and overdetermined boundary value problems.
{Forum Math.}{\ 3}(2)\ 143--156(1991).
\bibitem{Safoui}
G. A. Philippin and A, Safoui, Some applications of the maximum principle to a veriety of fully nonlienar elliptic PDE's. {Z. Angew. Math. Phys.}{\ 54}(5)\ 739--755(2003).
\bibitem{QX2017}
G. H. Qiu and  C. Xia, Overdetermined boundary value problems in $\mathbb S^n$.
{J. Math. Study}{\ 50}(2)\ 165--173(2017).
\bibitem{Reilly1974}
R. C. Reilly, Hessian of a function and curvatures of its graph.
{Mich. Math. J.}{\ 20}(4)\ 373--383(1974).
\bibitem{Reilly1980}
R. C. Reilly,  Geometric applications of the solvability of Neumann problems on a Riemannian manifold.
{Arch. Ration. Mech. Anal.}{\ 75}(1)\ 23--29(1980).
\bibitem{Serrin1971}
J. Serrin,  Symmetry problem in potential theory.
{Arch. Ration. Mech. Anal.}{\ 43}(4)\ 304--318(1971).
\bibitem{Schneider1993}R. Schneider, Convex bodies: the Brunn-Minkowski theory. Encyclopedia of Mathematics and its Applications, 44. Cambridge University Press, Cambridge, 1993. xiv+490 pp.
\bibitem{Souam2005}
R. Souam,  Schiffer's problem and an isoperimetric inequality for the first buckling eigenvalue of domains on S-2.
{Ann. Glob. Anal. Geom.}{\ 27}(4)\ 341--354(2005).
\bibitem{Tso1990}
K. S. Tso,  Remarks on critical exponents for Hessian operators.
{Ann. Inst. H. Poincare Anal. Non Lineaite}{\ 7}(2)\ 113--122(1990).
\bibitem{Bao2014}
B. Wang  and J. G. Bao, Mirror symmetry for a Hessian over-determined problem and its generalization.
{Commun. Pure Appl. Anal.}{\ 13}(6)\ 2305--2316(2014).
\bibitem{Weinberger1971}
H. F. Weinberger, Symmetry problem in potential theory - remark.
{Arch. Ration. Mech. Anal.}{\ 43}(4)\ 319--320(1971).
\end{thebibliography}
\end{document}